\documentclass[10pt,a4paper]{article}

\usepackage{lineno,hyperref}
\usepackage[utf8]{inputenc}
\usepackage[english]{babel}
\usepackage{amsmath}
\usepackage{amsfonts}
\usepackage{amsthm} 
\usepackage{bbm}
\usepackage{graphicx}
\usepackage{float}
\usepackage{hyperref}
\hypersetup{
    colorlinks=false,
    pdfborder={0 0 0},
} 
\usepackage{xcolor}
\usepackage[titletoc, title]{appendix}
\usepackage{natbib}
\usepackage{xr} 
\usepackage{setspace}
\usepackage{geometry}
\modulolinenumbers[5]
\numberwithin{equation}{section}

\newcommand{\bs}{\boldsymbol} 
\newcommand{\indep}{\rotatebox[origin=c]{90}{$\models$}}
\newcommand{\E}{\mathrm{E}}
\newcommand{\var}{\mathrm{Var}}
\newcommand{\m}{\text{-}}
\newtheorem{theorem}{Theorem}[section]
\newtheorem{corollary}{Corollary}[theorem]

\begin{document}
\setlength{\belowdisplayskip}{0pt} \setlength{\belowdisplayshortskip}{0pt}
\setlength{\abovedisplayskip}{0pt} \setlength{\abovedisplayshortskip}{0pt}

\begin{center}
\textbf{Inference for partial correlation when data are missing not at random}\\
Tetiana Gorbach*\footnote{Corresponding author email: tetiana.gorbach@umu.se}, Xavier de Luna*\\
*Department of Statistics, USBE, Ume\aa \  University, SE-90187, Ume\aa, Sweden
\end{center}

\textbf{Abstract}
We introduce uncertainty regions to perform inference on partial correlations when data are missing not at random. These uncertainty regions are shown to have a desired asymptotic coverage. Their finite sample performance is illustrated via simulations and real data example.


\section{Introduction}
This paper proposes methods to perform inference on partial correlations when data are missing not at random. The motivation for this work comes from a recent investigation of the relationship between longitudinal changes in brain structure, e.g. gray matter volume of hippocampus, and changes in cognition, e.g. episodic memory, when adjusting for the effect of age and hypertension, see \citet{Gorbach2017167}.
A partial correlation coefficient may be used to describe an association between two random variables, such as changes in brain and cognition, that is not due to other related covariates, for example age and hypertension (see \citet{anderson1958introduction} for theory and \citet{Nilsson1997theBetula}, \citet{marellec2006228}, \citet{vanpetten2004memory} for application).
However, a natural feature of most longitudinal investigations is the occurrence of missing data
due to dropout. In \citet{Gorbach2017167}, for example, measures of the episodic memory change could be obtained for each individual, while  41\%
of data on the gray matter volume changes was missing due to dropout. Moreover, individuals with more pronounced health and brain deterioration are expected to drop out from longitudinal investigations earlier than healthier subjects. Thus the probability of an observation to be missing is expected to depend on its unobserved value (data missing not at random).  

Some work has been devoted to inference on partial correlation when data are missing. \citet{dAngelo2012missing}, for example, considered an EM algorithm and multiple imputation for the case of trivariate normal distribution with data missing at random. \cite{Gorbach2017167} inferred on statistical significance of partial correlation allowing for data missing not at random. This was done using the relationship between the partial correlation and a regression parameter in combination with results developed by \citet{Genback2015}. This approach does not allow, however, to construct uncertainty regions for partial correlations but only to perform significance testing.  

In this paper we consider the situation when the data is missing not at random. To model the dependency between variables (which are not observed for all individuals) and the probability that their observations are missing, we introduce a parameter $\gamma$ which is typically unknown in applications.  We then construct confidence intervals (with given coverage (1-$\alpha$)100\%) for a partial correlation of interest for each plausible value of the parameter $\gamma$ based on asymptotic results. We then propose to use the union of these confidence intervals, which we call uncertainty region, and prove that this region has at least (1-$\alpha$)100\% coverage asymptotically.

This paper is organized as follows. Partial correlation and its relation to a regression parameter are briefly introduced in \autoref{Partial correlation section}. The missing data mechanisms we consider are described in \autoref{Missing data mechanisms section}. \autoref{Inference section} presents our results and introduces uncertainty regions for each missing data mechanisms, while \autoref{Application section} illustrates the application of the method to the aforementioned longitudinal study of the relation between changes in brain structure and cognition. A simulation study is conducted in \autoref{Simulation study section}, followed by concluding remarks in \autoref{Discussion section}. 

\section{Partial correlation}
\label{Partial correlation section}
The partial correlation between random variables $X_1$ and $X_2$ while adjusting for $X_3, \ldots, X_p$ is defined as the correlation between residuals of the projections of $X_1$ and $X_2$ on the linear space spanned by $X_3, \ldots, X_p$. 
Let $X_1, \ldots, X_p$  be random variables with finite second moments, $EX_j^{2}<~\infty,$ $j=1,\ldots,p$,
and consider the projections of $X_1$ and $X_2$ on the linear spaces spanned by $X_2, \ldots, X_p$ and $X_3, \ldots, X_p$ respectively: 
\begin{align}\label{model: response}
\begin{split}
& X_1=\beta_1+\beta_2X_2+\ldots+\beta_{p} X_p+\xi_1,\\ 
& X_2=\theta_1+\theta_2X_3+\ldots+\theta_{p-1} X_p+\xi_2,
\end{split}
\end{align}
where  $\E\xi_i=0$,  
$\text{Cov}(\xi_i,{X}_{j})=0, \ i=1,2, \ j=3,\ldots,p$ and
$\text{Cov}(\xi_1,X_2)=0.$ We will assume that $\sigma^2_{2.3\ldots p}=\var(\xi_2) \neq 0$ and $\sigma^2_{1.2\ldots p}=\var(\xi_1) \neq 0.$
 From (\ref{model: response}) we have that
$X_1=(\beta_1+\beta_2\theta_1)+(\beta_2\theta_2+\beta_3)X_3+\ldots+(\beta_2\theta_{p-1}+\beta_p) X_p+\beta_2\xi_2+\xi_1$
and $
 \text{Cov}(\xi_1,\xi_2)=0$.

The partial correlation $\rho$ between  $X_1$ and $X_2$ while adjusting for $X_3,\ldots, X_p$  is then
$ \rho=\text{Corr}(\xi_2,\beta_2\xi_2+\xi_1)$ 
 and can be expressed as 
\begin{align}
\rho=\frac{\beta_2}{\sqrt{\beta_2^2+\sigma^2_{1.2\ldots p}/\sigma^2_{2.3\ldots p}}}\label{dependency between partial correlation and regression parameter}. 
 \end{align}
We use representation (\ref{dependency between partial correlation and regression parameter}) in the sequel.

\section{Missing data mechanisms (MDM)}
\label{Missing data mechanisms section}
We consider three models where the data on $X_1$ or $X_1$ and $X_2$ are missing not at random (MNAR).
\begin{center}
\begin{figure}[H]
\begin{center}
\includegraphics{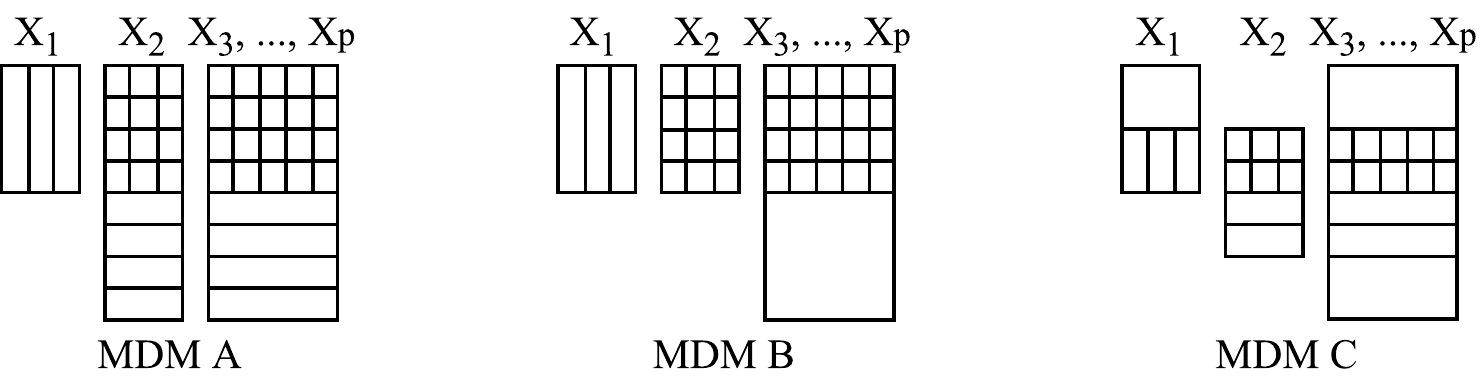}
\end{center}
\caption{Data patterns. Vertical shading represents data used for estimation of $\beta_2$ and $\sigma^2_{1.2\ldots p}$, horizontal shading represents data used for estimation of $\sigma^2_{2.3\ldots p}$. Grid shading corresponds to data used for estimation of  $\beta_2$, $\sigma^2_{1.2\ldots p}$ and $\sigma^2_{2.3\ldots p}.$ Empty boxes represent the observed data that is not used in the estimation procedures.}
\label{Figure:missing mechanisms}
\end{figure}
\end{center}
\noindent \textbf{MDM A.} Let $X_2, X_3, \ldots, X_p$ be fully observed while $X_1$ is observed if $Z=1$ and missing otherwise  (see \autoref{Figure:missing mechanisms}), where 
\begin{equation}\label{model missing mechanism A} 
Z=\mathbbm{1}(\bs{X}_{\m1}\bs{\delta}+\eta_1>0), 
\end{equation}  $\bs{X}_{\m1}=(1, X_2,\ldots, X_p)$, $\bs{\delta}$ is a  $p$ column vector of unknown parameters, $\eta_1 \sim \mathcal{N}(0,1)$ and $\mathbbm{1}$ is the indicator function.
In order to introduce missing not at random data in $X_1$ we follow \cite{Genback2015} by modeling $\xi_1$ in (\ref{model: response}) as 
 $\xi_1=~\gamma\sigma_{1.2\ldots p}\eta_1 + \epsilon,$ where $\E\epsilon=0,$  $\eta_1 \indep  (X_2,\ldots, X_p),$ $\epsilon \indep ( X_2,\ldots, X_p,\eta_1)$ and $\indep$ denotes independence between random variables. Then $\gamma \neq 0$ corresponds to MNAR data, while data are missing at random (MAR) when $\gamma=0.$\\
\noindent \textbf {MDM B.} Let $X_3, \ldots, X_p$ be fully observed while $X_{1}$ and $X_2$ are observed if $Z=1$ and missing otherwise (see  \autoref{Figure:missing mechanisms}), where
\begin{equation}\label{model missing mechanism B} 
Z=\mathbbm{1}(\bs{X}_{\m12}\bs{\delta}+\eta_1>0), 
\end{equation}  
$\bs{X}_{\m12}=(1, X_3,\ldots, X_p),$
$\bs{\delta}$ is  a $(p-1)$  column vector of unknown parameters, $\eta_1 \sim \mathcal{N}(0,1).$ We introduce missing not at random data in $X_1$ by modeling $\xi_1$ in (\ref{model: response}) as 
 $\xi_1=~\gamma\sigma_{1.2\ldots p}\eta_1 + \epsilon,$ where $\E\epsilon=0,$ $\eta_1 \indep (X_2,\ldots, X_p),$ $\epsilon \indep (X_2,\ldots, X_p,\eta_1)$, $\xi_2 \indep (X_3,\ldots, X_p, \eta_1).$  Then $\gamma \neq 0$ corresponds to MNAR data for $X_1$, while data on $X_1$ are MAR when $\gamma=0$. \\
\noindent \textbf{MDM C.}
Let $ X_3, \ldots, X_p$ be fully observed while $X_1$ is observed if $Z_1=1$ and missing otherwise, $X_2$ is observed if $Z_2=1$ and missing otherwise (see  \autoref{Figure:missing mechanisms}), where
  \begin{align}\label{model dropout: separate, model dropout1}
&Z_1=\mathbbm{1}(\bs{X}_{\m12}\bs{\delta}_1+\eta_1>0),\\
\label{model dropout: separate, model dropout2}
&Z_2=\mathbbm{1}(\bs{X}_{\m12}\bs{\delta}_2+\eta_2>0), 
\end{align}
$\bs{X}_{\m12}=(1, X_3,\ldots, X_p),$  $\bs{\delta}_1$  and $\bs{\delta}_2$ are $(p-1)$ column vectors of unknown parameters and $(\eta_1,\eta_2) \sim \mathcal{N}(\bs{0}, \bs{I}_2)$, $\bs{I}_2$ is an identity matrix of size 2. As above we introduce missing not at random data by modeling $\xi_1$ and $\xi_2$ in (\ref{model: response}) as $\xi_1=\gamma_{1}\sigma_{1.2\ldots p}\eta_1+\epsilon_1$, $\xi_2=\gamma_{2}\sigma_{2.3\ldots p}\eta_2+\epsilon_2$, where $\E\epsilon_1=0$, $\E\epsilon_2=0,$ $\eta_1\indep (X_2,\ldots, X_p, \eta_2),$ $\epsilon_1 \indep (X_2,\ldots, X_p,\eta_1,\eta_2),$ $\eta_2 \indep (X_3,\ldots, X_p)$ , $\epsilon_2 \indep (X_3,\ldots, X_p,\eta_2).$ $\gamma_1 \neq 0$ corresponds to MNAR data for $X_1$ and $\gamma_2 \neq 0$ corresponds to MNAR data for $X_2.$

\section{Inference}
\label{Inference section}
\subsection{Inference under MDM A}
Let $\{X_{1i}, X_{2i}, \ldots, X_{pi},Z_i\}_{i=1}^{N}$  be a random sample from $(X_{1}, X_{2}, \ldots, X_{p},Z)$ for which MDM A holds. For a given $\gamma,$ we  propose an estimator $\widehat{\rho}_{\gamma}$ for $\rho$ based on bias correction of complete cases ordinary least squares (OLS) estimators of quantities in (\ref{dependency between partial correlation and regression parameter}) (see \autoref{Figure:missing mechanisms}):
\begin{equation}\label{part cor estimator}
\widehat{\rho}_{\gamma}=\frac{\widehat{\beta}_2}{\sqrt{\widehat{\beta}_2^2+\widehat{\sigma}_{1.2 \ldots p}^2/ \widehat{\sigma}^2_{2.3\ldots p}}},
\end{equation}
where  
\begin{align}
&\widehat{\beta}_2=\widehat{\beta}_{2,ols}-\gamma\widehat{\sigma}_{1.2 \ldots p}\left[(\bs{X}_{\m1s}^T\bs{X}_{\m1s})^{\m1}\bs{X}_{\m1s}^T\bs{\lambda}_{\widehat{\bs{u}}}\right]_2,\nonumber \\
&\widehat{\sigma}_{1.2 \ldots p}^2=\frac{\widehat{\sigma}_{1.2 \ldots p, ols}^2}{1+\gamma^2(\bs{\widehat{u}}^{T}\bs{\lambda}_{\widehat{\bs{u}}}-\bs{\lambda}_{\widehat{\bs{u}}}^{T}\bs{X}_{\m1s} (\bs{X}_{\m1s}^T\bs{X}_{\m1s})^{\m1} \bs{X}_{\m1s}^T \bs{\lambda}_{\widehat{\bs{u}}})/(n-p) } \label{eq:corrected var},
\end{align}
and $\widehat{\sigma}^2_{2.3\ldots p}$ is an OLS estimator of $\sigma^2_{2.3\ldots p}$ based on $\{ X_{2i}, \ldots, X_{pi}\}_{i=1}^{N}.$  \\ Here $\widehat{\beta}_{2,ols}=\left[(\bs{X}_{\m1s}^T\bs{X}_{\m1s})^{\m1}\bs{X}_{\m1s}^T\bs{X}_{1s}\right]_2$ and $\widehat{\sigma}_{1.2 \ldots p, ols}^2=\bs{X}_{1s}^T(\bs{I}_n-\bs{X}_{\m1s}(\bs{X}_{\m1s}^T\bs{X}_{\m1s})^{\m1}\bs{X}_{\m1s}^T)\bs{X}_{1s}/(n-p)$ are OLS estimators of $\beta_2$ and $\sigma_{1.2\ldots p}^{2}$ based on $n$ complete cases; $\bs{X}_{1s}$ denotes an $n<N$ vector of observed $X_{1}$ for complete cases; $\bs{X}_{\m1s}$ represents an $n \times p$ matrix of observed covariates $(1,X_2,\ldots,X_p)$ for complete cases.  $\bs{\widehat{u}}=(\widehat{u}_1, \ldots,\widehat{u}_n)^{T} =-\bs{X}_{\m1s}\widehat{\bs{\delta}}$, $\widehat{\bs{\delta}}$ is the maximum likelihood estimator of $\bs{\delta}$ in probit model (\ref{model missing mechanism A}) based on full data $\{X_{2i}, \ldots, X_{pi},Z_i\}_{i=1}^{N}$; $\bs{\lambda_{\widehat{u}}}=(\lambda(\widehat{u}_1), \ldots, \lambda(\widehat{u}_n))^{T}$, $\lambda(\widehat{u}_i)=\frac{\phi(\widehat{u}_i)}{\Phi(-\widehat{u}_i)}, \  i=1,\ldots,n$ denotes the inverse Mills ratio, $\phi$ and $\Phi$ are respectively the standard normal density and cumulative distribution functions. Also, $\bs{I}_n$ is an $n \times n$ identity matrix and $[v]_2$ denotes the second element of a vector $v.$ 
\begin{theorem}\label{theorem CAN}
Under MDM A and regularity assumptions (see \autoref{app:A}) $\widehat{\rho}_{\gamma}$ is a consistent estimator of $\rho$ and
$$\sqrt{n}(\widehat{\rho}_{\gamma}-\rho)/\widehat{\text{se}}_{\widehat{\rho}_{\gamma}} \rightarrow_d \mathcal{N}(0, 1), $$
where 
$
\widehat{\text{se}}_{\widehat{\rho}_{\gamma}}=\sqrt{\frac{\widehat{\sigma}_{1.2 \ldots p}^{2}(1+\gamma^{2}\bs{\widehat{u}}^{T}\bs{\lambda}_{\bs{\widehat{u}}}/n-\gamma^{2}\bs{\lambda}_{\bs{\widehat{u}}}^{T}\bs{\lambda}_{\bs{\widehat{u}}}/n) (\bs{X}_{\m1s}^T\bs{X}_{\m1s})_{22}^{\m1}}{\widehat{\beta}^2_2+\widehat{\sigma}^2_{1.2 \ldots p}/\widehat{\sigma}^2_{2.3\ldots p}}}.$
\end{theorem} 

\noindent A proof is provided in \autoref{app:A}. A $(1-\alpha)100\%$ confidence interval  for $\rho$ is thus:
\begin{equation} \label{CI}
\text{CI}(\rho, \gamma,\alpha)=[\widehat{\rho}_{\gamma}-c_{\frac{\alpha}{2}}\widehat{\text{se}}_{\widehat{\rho}_{\gamma}};\widehat{\rho}_{\gamma}+c_{\frac{\alpha}{2}}\widehat{\text{se}}_{\widehat{\rho}_{\gamma}}].
\end{equation}
Here $c_{\frac{\alpha}{2}}$ is the  $(1-\alpha)100\text{th}$  percentile of the standard normal distribution.
However, the true value of $\gamma,$ $\gamma_0,$ is typically unknown in applications. Setting it to one certain value in analysis, for example 0, is a strong assumption which is typically difficult to check empirically. Instead we propose to assume that $\gamma_0$ belongs to an interval $[\gamma_{min}, \gamma_{max}]$ and provide inference under this weaker assumption.  For example, since $\gamma$  is the correlation between $\xi_1$ and $\eta_1,$ $\gamma_0 \in [-1, 1].$  Then, intervals (\ref{CI}) can be constructed for each $\gamma \in [\gamma_{min}, \gamma_{max}]$. Although each specific confidence interval (\ref{CI}) may fail to cover the true $\rho_0$ corresponding to $\gamma_0$ with probability $(1-\alpha)100\%,$ their union
\begin{equation*}
\text{UR}(\rho,[\gamma_{min}, \gamma_{max}], \alpha )=\bigcup\limits_{ \gamma \in [\gamma_{min}, \gamma_{max}]} \text{CI}(\rho, \gamma,\alpha),
\end{equation*} 
which we call the  uncertainty region for $\rho$, covers  $\rho_0$  with  at least $(1-\alpha)100\%$ probability as can be seen below.
\begin{corollary}
\label{corollary: coverage}
Under the assumptions of \autoref{theorem CAN}, if the true $\gamma_0 \in [\gamma_{min}, \gamma_{max}]$, the uncertainty region $\text{UR }(\rho,[\gamma_{min}, \gamma_{max}], \alpha)$  has asymptotic coverage for $\rho_0$  of at least $(1-\alpha)100\%.$
\end{corollary}
\noindent A proof is provided in \autoref{app:A}. 
\subsection{Inference under MDM B}
Results under missing mechanisms B follow the same structure as for mechanism A. A consistent and asymptotically normal estimator for partial correlation (see \autoref{app:B}, \autoref{theorem: CAN for rho mechanism B})  is defined by (\ref{part cor estimator}), with $\bs{\widehat{u}}=(\widehat{u}_1, \ldots,\widehat{u}_n)^{T} =-\bs{X}_{\m12s}\widehat{\bs{\delta}}$, where $\bs{X}_{\m12s}$ represents an $n \times (p-1)$ matrix of observed covariates $(1, X_3, \ldots, X_p)$ for complete cases; $\widehat{\bs{\delta}}$ a maximum likelihood estimator of $\bs{\delta}$ in probit model (\ref{model missing mechanism B}) based on full data $\{X_{3i}, \ldots, X_{pi},Z_i\}_{i=~1}^{N},$ and $\widehat{\sigma}^2_{2.3\ldots p}$ is an ordinary OLS estimator of $\sigma^2_{2.3\ldots p}$ based on complete cases $\{X_{2i}, \ldots, X_{pi}\}_{i=1}^{n}$. Uncertainty regions can then be constructed as above for MDM A (see \autoref{app:B}, \autoref{corollary: coverage B}).
\subsection{Inference under MDM C}
Let $\{X_{1i}, X_{2i}, \ldots, X_{pi},Z_{1i},Z_{2i}\}_{i=1}^{N}$ be a random sample from $(X_{1}, X_{2}, \ldots, X_{p},Z_1,Z_2)$ for which MDM C holds.  
Results under MDM C follow the same structure as for mechanism A. A consistent and asymptotically normal estimator for $\rho$ (see \autoref{app:B}, \autoref{theorem: CAN for rho mechanism C} for proofs) is defined as
\begin{equation*}
\widehat{\rho}_{\gamma_1, \gamma_2}=\frac{\widehat{\beta}_2}{\sqrt{\widehat{\beta}_2^2+\widehat{\sigma}_{1.2 \ldots p}^2/\widehat{\sigma}_{2.3 \ldots p}^{2}}},
\end{equation*}
where
\begin{align*}
&\widehat{\beta}_2=\widehat{\beta}_{2,ols}-\gamma\widehat{\sigma}_{1.2 \ldots p}\left[(\bs{X}_{\m1s}^T\bs{X}_{\m1s})^{\m1}\bs{X}_{\m1s}^T\bs{\lambda}_{\widehat{\bs{u}}}\right]_2,\\
&\widehat{\sigma}_{1.2 \ldots p}^2=\frac{\widehat{\sigma}_{1.2 \ldots p, ols}^2}{1+\gamma_1^2(\bs{\widehat{u}}^{T}\bs{\lambda}_{\widehat{\bs{u}}}-\bs{\lambda}_{\widehat{\bs{u}}}^{T}\bs{X}_{\m1s} (\bs{X}_{\m1s}^T\bs{X}_{\m1s})^{\m1} \bs{X}_{\m1s}^T \bs{\lambda}_{\widehat{\bs{u}}})/(n-p) },\\
&\widehat{\sigma}_{2.3\ldots p}^2=\frac{\widehat{\sigma}_{2.3\ldots p, ols}^2}{1+\gamma_2^2 (\bs{\widehat{w}}^{T}\bs{\lambda}_{\widehat{\bs{w}}}-\bs{\lambda}_{\widehat{\bs{w}}}^{T}\bs{X}_{\m12s_2} (\bs{X}_{\m12s_2}^T\bs{X}_{\m12s_2})^{\m1} \bs{X}_{\m12s_2}^T \bs{\lambda}_{\widehat{\bs{w}}})/(n_2-p)},
\end{align*}
$\widehat{\beta}_{2,ols}=\left[(\bs{X}_{\m1s}^T\bs{X}_{\m1s})^{\m1}\bs{X}_{\m1s}^T\bs{X}_{1s}\right]_2$ and $\widehat{\sigma}_{1.2 \ldots p, ols}^2=\bs{X}_{1s}^T(\bs{I}_n-\bs{X}_{\m1s}(\bs{X}_{\m1s}^T\bs{X}_{\m1s})^{\m1}\bs{X}_{\m1s}^T)\bs{X}_{1s}/(n-p)$ are OLS estimators of $\beta_2$ and $\sigma_{1.2\ldots p}^{2}$ based on $n$ complete cases, and \\
$\widehat{\sigma}_{2.3\ldots p, ols}^2=\bs{X}_{2s_2}^T(\bs{I}_{n_2}-\bs{X}_{\m12s_2}(\bs{X}_{\m12s}^T\bs{X}_{\m12s_2})^{\m1}\bs{X}_{\m12s_2}^T)\bs{X}_{2s_2}/(n_2-p)$  is an OLS estimator of $\sigma_{2.3\ldots p}^{2}$ based on $n_2$ cases with observed $X_2$. $\bs{X}_{\m1s}$ and $\bs{X}_{\m12s}$ represent respectively an $n \times p$ and an $n \times (p-1)$ matrices of observed covariates $(1, X_2, \ldots, X_p)$ and $(1, X_3, \ldots, X_p)$ for complete cases; $\bs{X}_{\m12s_2}$ is an $n_2 \times (p-1)$ matrix of observed covariates $(1,X_3,\ldots, X_p)$ for cases with observed $X_2$. $\bs{X}_{1s}$ denotes an $n$ vector of observed $X_{1}$ for complete cases, $\bs{X}_{2s_2}$ is an $n_2$ vector of observed $X_{2}$. $\bs{\lambda_{\widehat{u}}}=(\lambda(\widehat{u}_1), \ldots, \lambda(\widehat{u}_n))^{T},$ where $\lambda$ denotes inverse Mills ratio,
$\bs{\widehat{u}}=(\widehat{u}_1, \ldots,\widehat{u}_n)^{T} =-\bs{X}_{\m12s}\widehat{\bs{\delta}}_1$, $\widehat{\bs{\delta}}_1$ is the maximum likelihood estimator of $\bs{\delta}_1$ under  model (\ref{model dropout: separate, model dropout1}) based on full data $\{X_{3i}, \ldots, X_{pi},Z_{1i}\}_{i=~1}^{N}.$  
    $\bs{\lambda_{\widehat{w}}}=(\lambda(\widehat{w}_1), \ldots, \lambda(\widehat{w}_{n_2}))^{T}$, where $\lambda$ denotes inverse Mills ratio, $\bs{\widehat{w}}=(\widehat{w}_1, \ldots,\widehat{w}_{n_2})^{T} =-\bs{X}_{\m12s_2}\widehat{\bs{\delta}}_2$, $\widehat{\bs{\delta}}_2$  is the maximum likelihood estimator of $\bs{\delta}_2$ under probit model (\ref{model dropout: separate, model dropout2}) based on full data $\{X_{3i}, \ldots, X_{pi},Z_{2i}\}_{i=~1}^{N}.$
A $(1-\alpha)100\%$ confidence interval is
\begin{equation*} 
\text{CI}(\rho, \gamma_1,\gamma_2, \alpha)=[\widehat{\rho}_{\gamma_1, \gamma_2}-c_{\frac{\alpha}{2}}\widehat{\text{se}}_{\widehat{\rho}_{\gamma_1, \gamma_2}};\widehat{\rho}_{\gamma_1, \gamma_2}+c_{\frac{\alpha}{2}}\widehat{\text{se}}_{\widehat{\rho}_{\gamma_1, \gamma_2}}].
\end{equation*}
Here $c_{\frac{\alpha}{2}}$ is the  $(1-\alpha)100\%$  percentile of the standard normal distribution and 
\begin{align*}
\widehat{\text{se}}_{\widehat{\rho}_{\gamma_1, \gamma_2}}=\sqrt{\frac{\widehat{\sigma}_{1.2 \ldots p}^{2}(1+\gamma_1^{2}\bs{\widehat{u}}^{T}\bs{\lambda}_{\bs{\widehat{u}}}/n-\gamma_1^{2}\bs{\lambda}_{\bs{\widehat{u}}}^{T}\bs{\lambda}_{\bs{\widehat{u}}}/n) (\bs{X}_{\m1s}^T\bs{X}_{\m1s})_{22}^{\m1}}{\widehat{\beta}_2^2+\widehat{\sigma}_{1.2 \ldots p}^2/\widehat{\sigma}_{2.3 \ldots p}^{2}}}.
\end{align*} 
An uncertainty region is, accordingly, 
\begin{equation*}
\text{UR}(\rho,[\gamma_{1min}, \gamma_{1max}],[\gamma_{2min}, \gamma_{2max}], \alpha )=\bigcup\limits_{ \gamma_1 \in [\gamma_{1min}, \gamma_{1max}],\gamma_2 \in [\gamma_{2min}, \gamma_{2max}]} \text{CI}(\rho,\gamma_1, \gamma_2, \alpha),
\end{equation*}   
and asymptotically covers the true $\rho=\rho_{\gamma_{10}, \gamma_{20}}$, where  $(\gamma_{10}, \gamma_{20})$ is the true value of $(\gamma_{1}, \gamma_{2})$,  with  at least $(1-\alpha)100\%$  probability (see \autoref{app:B}, \autoref{corollary: coverage C}).

\section{Application}
\label{Application section}
We use the theoretical results  developed in this paper to infer on partial correlation between longitudinal changes in gray matter volume of hippocampus and episodic memory decline, when adjusting for the effect of age and hypertension \citep{Gorbach2017167}, with  the data from the Betula study \citep{Nilsson1997theBetula}. Briefly, the sample consists of 264 older adults that underwent magnetic resonance imaging (MRI) at one of the Betula waves, had up to 25 years history of cognitive assessment and were scheduled for a MRI follow-up. Of the 264 initially scanned
participants, 155 underwent a follow-up MRI
examination; see  \citet{Gorbach2017167} for more detailed description of the sample and measures used. Since information on cognition changes could be obtained for all individuals while changes in gray matter volume of hippocampus are missing for approximately 41\% of individuals in the sample, we consider missing data mechanism A.  
Point estimates $\widehat{\rho}_{\gamma}$ can be constructed under assumptions of the true $\gamma_0=\gamma$,  where $\gamma \in [0;1]$. $\gamma$ is constrained to be nonnegative, since given age, hypertension and cognition change, an individual with smaller value of gray matter change, that is fraction of gray matter volume at second and first measurements,   may be expected to have poorer health and thus more likely to drop out, which corresponds to $\gamma\geq0$.  
As can be seen from \autoref{Figure: application},  an uncertainty region for the partial correlation is $UR(\rho, [0;1], 0.05)=[0.153; 0.511].$ The interval does not contain 0 which is in line with the analysis in \citet{Gorbach2017167} where an uncertainty region for $\rho$ could not be produced.

\begin{figure}[h]
\begin{center}
\includegraphics[]{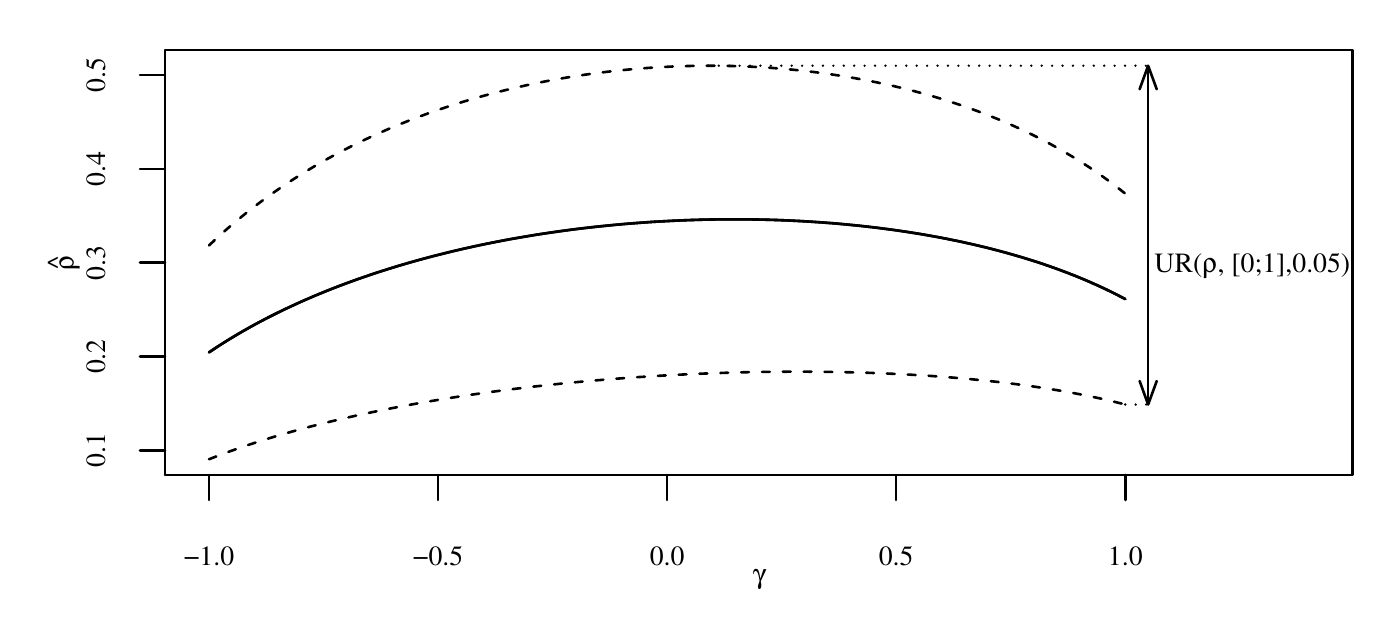}
\end{center}
\caption{Estimated partial correlation between changes of gray matter volume of hippocampus and episodic memory changes controlled for age and hypertension (solid line) for different values of parameter $\gamma$ with corresponding 95\% CI's (bounds of CI's are represented as dashed curves) in presence of missing not at random data.}
\label{Figure: application}
\end{figure}

\section{Simulation study}
\label{Simulation study section}
This simulation study uses a design inspired by the above application.
Observations for age ($X_3$) and hypertension ($X_4$) are simulated from the empirical distribution of the data. Since in the study episodic memory change ($X_2$) was available for the full sample, while hippocampus gray matter change ($X_1$) was partially observed, we simulate data under missing mechanism A as follows:
\begin{align*}
X_{2i}&=2.313-0.042X_{3i}-0.216X_{4i}+\xi_{2i},\\
X_{1i}&=1.092+0.01X_{2i}-0.002X_{3i}-0.006X_{4i}+0.028\gamma_0 \eta_{1i}+\epsilon_i,\\
Z_i &=\mathbbm{1}(2.708+0.548X_{2i}-0.036X_{3i}-0.042X_{4i}+\eta_{1i}>0),\\
(\xi_{2i},\epsilon_i,  \eta_{1i})& \sim N_3\left(\bs{0},\text{diag} (1.16,0.028^2(1-\gamma^2), 1)\right), i=1,\ldots,N,
\end{align*}
where regression parameters for simulation of $X_{2i}$ and $X_{1i}$ are the corresponding OLS estimates from complete cases linear regressions fit obtained from the data. $Z_i$ is simulated from the probit regression fit to the data where we have changed the parameter for $X_{2i}$ from 0.048 to 0.548 to increase the difference between complete cases and full data distributions.  
 The partial correlation between simulated hippocampus change and episodic memory decline is $\rho=0.359.$
Data are generated for $\gamma_0=0.1, 0.5, 0.8$ and for sample sizes $N=100$ and $N=250$. Around 50\% of data are missing for each generated sample.  The width and empirical coverage of 95\%  confidence intervals based on complete cases (CC CI), confidence intervals constructed under the true data law $\gamma=\gamma_0$ (oracle CI) and uncertainty regions $UR(\rho, [0, 0.5], 0.05)$ are computed for 1000 replicates.  
\begin{figure}[H]
\begin{center}
\includegraphics{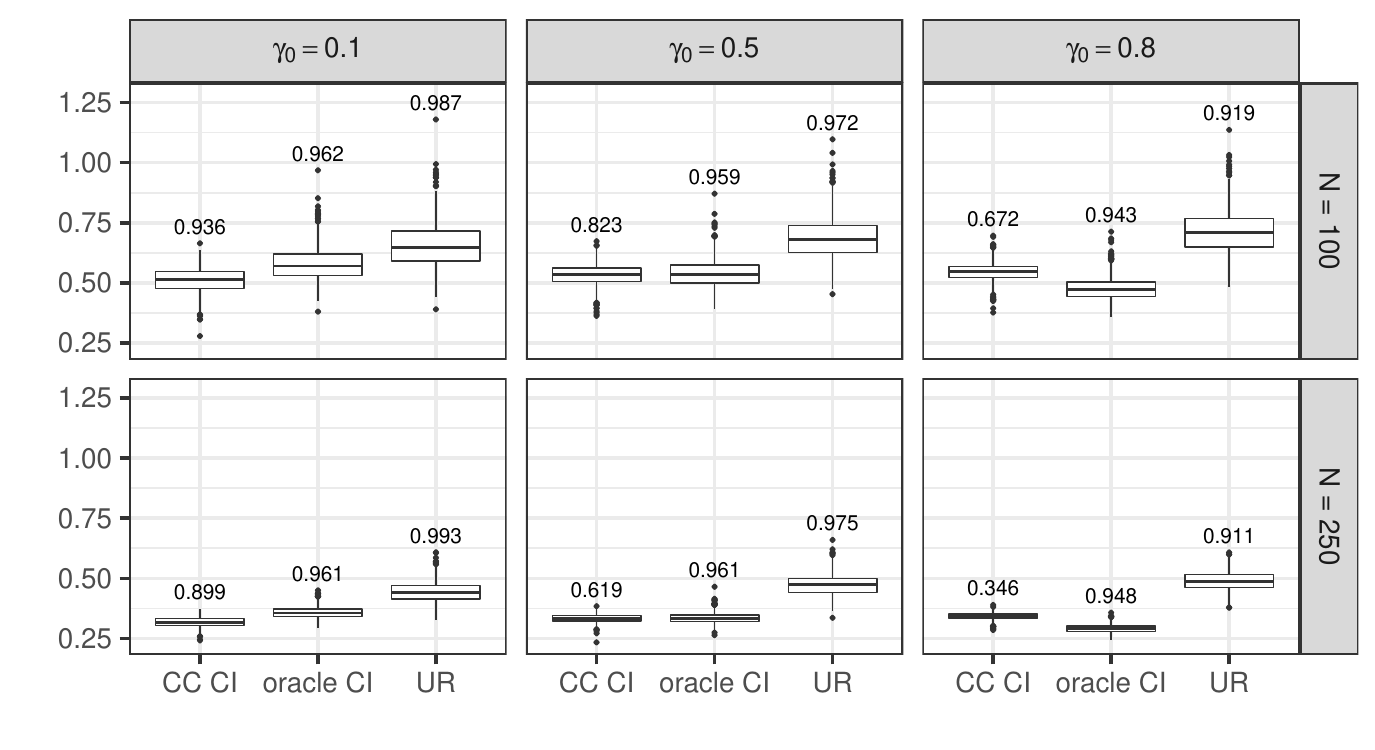}
\caption{Boxplots of widths of the intervals for 95\% nominal coverage for complete cases confidence intervals (CC CI), CI  under the true data law ($\gamma=\gamma_0$, oracle CI) and URs assuming that $\gamma \in [0,0.5]$  for 1000 simulations of data using $\gamma_0=0.1, 0.5, 0.8$ and for sample sizes $N=100$ and $N=250$.  Empirical coverage is labeled above each boxplot.}
\label{Figure:simulation results}
\end{center}
\end{figure}

As \autoref{Figure:simulation results} shows, the empirical coverage of complete cases confidence intervals decreases with increasing value of $\gamma$ and/or increasing sample size.  The empirical coverage of oracle confidence intervals based on true data law is around 95\% as expected. The empirical coverage of constructed uncertainty regions, in turn, is as expected above 95\% when  $\gamma_0$  used for data generation lies within the assumed in estimation interval $[0, 0.5]$ (for the cases $\gamma_0=0.1$ and $\gamma_0=0.5$). Noteworthy, even for the data generated under $\gamma_0=0.8$, uncertainty regions constructed under the assumption of $\gamma_0 \in [0, 0.5]$ have much higher empirical coverage than complete cases confidence intervals.

\section{Discussion}
\label{Discussion section}
The uncertainty regions proposed here are an alternative to establishing possible identifiability of $\gamma$ and the partial correlation in the considered semiparametric missing mechanism models.  Known methods for estimation in this constext are Heckman two-step type approaches \citep{heckman1979sample,Vella1998Estimating}, which rely heavily on the nonlinearity of the inverse Mills ratio. However, as the inverse Mills ratio is linear for a wide range of its arguments, identifiability and thus point estimation is not possible in practice \citep{Puhani2000Heckman}.

 Uncertainty regions were studied in wider generality in \citet{vansteelandt2006ignorance}, who proposed to construct an uncertainty region for an unidentified parameter by adding confidence limits to estimated bounds of an ignorance region, which is a range of parameter values that correspond to different full data distributions compatible with the observed data law. To do so, \citet{vansteelandt2006ignorance}  relies on the assumption of lower and upper bounds of the ignorance region being independent of the observed data law (Assumption 2, p. 960). In our approach, by using instead a union of confidence intervals to define an uncertainty region, one avoids the aforementioned assumption, which is seldom fulfilled.
  
  \citet{Genback2015} uses bounds for the variance of the residuals in a linear regression to deduce uncertainty regions for regression parameters when data is missing not at random. The bias corrected estimators of such residual variance introduced in this paper (e.g. (\ref{eq:corrected var})) can be used to provide narrower uncertainty regions than those proposed in \citet{Genback2015}.
  

Finally, note that the results developed in the paper also hold when missing data occur in $X_3, \ldots, X_p$ if the latter is missing at random. 


\section*{Acknowledgements}
The authors would like to thank Minna Genb\"ack  and Angel Angelov for their valuable comments. This work was supported by Swedish Research Council (grant number 340-2012-5931 to Xavier de Luna). 
\begin{appendices}
\section{}
\label{app:A}
\noindent \textbf{Regularity assumptions for \autoref{theorem CAN}.}
\begin{enumerate}
 \item \label{ThA:1}$EX_j^{2}<~\infty, \  j=1,\ldots,p;$ $\sigma^2_{2.3\ldots p}\neq 0$, $\sigma^2_{1.2\ldots p}\neq 0;$
 \item \label{ThA:2}$ \det \bs{X}_{\m1s}^T\bs{X}_{\m1s}\neq0$, $ \det \E(\bs{X}_{\m1}^{T} \bs{X}_{\m1}|Z=1) \neq0;$ 
 \item \label{ThA:3}$\det\bs{X}_{\m12s}^T\bs{X}_{\m12s}\neq0, \  \det \E(\bs{X}_{\m12}^{T} \bs{X}_{\m12}) \neq0,$ where $\bs{X}_{\m12s}=(\bs{1}, \bs{X}_3, \ldots, \bs{X}_p)$ - $N \times (p-1)$ matrix of observed covariates and $\bs{X}_{\m12}=(1, X_3, \ldots, X_p)$ is a vector of random variables;
 \item \label{ThA:4} $1+\frac{\gamma^2}{n-p}(\bs{\widehat{u}}^{T}\bs{\lambda}_{\widehat{\bs{u}}}-\bs{\lambda}_{\widehat{\bs{u}}}^{T}\bs{X}_{\m1s} (\bs{X}_{\m1s}^T\bs{X}_{\m1s})^{\m1} \bs{X}_{\m1s}^T \bs{\lambda}_{\widehat{\bs{u}}}) \neq 0$ for all $ n$, \\$
 1+\gamma^2\E(u\lambda_u|Z=1)-
 \gamma^2\E\left(\bs{X}_{\m1}\lambda_u|Z=1\right)\E^{\m1}(\bs{X}_{\m1}^{T}\bs{X}_{\m1}|Z=1)\E(\bs{X}_{\m1}^{T}\lambda_u|Z=1)\neq 0,$
  where $u=-\bs{X}_{\m1}\bs{\delta}$ and $\lambda_u=\frac{\phi(u)}{\Phi(- u)}$, $\lambda$ denotes inverse Mills ratio;
 \item \label{ThA:5} $\widehat{\sigma}^2_{2.3\ldots p}\neq 0$, and at least one of $\widehat{\beta}_2 \neq 0$ or $\widehat{\sigma}_{1.2 \ldots p}^2 \neq 0$ for all $ n.$
 \end{enumerate}
 \begin{proof}[Proof of \autoref{theorem CAN}.]
$\widehat{\sigma}^2_{2.3\ldots p}\rightarrow_p \sigma^2_{2.3\ldots p} $ as an OLS estimator of $\sigma^2_{2.3\ldots p}$ based on $\{X_{2i}, \ldots, X_{pi}\}_{i=1}^{N}$ and from regularity assumptions \ref{ThA:1}, \ref{ThA:3}.

\textbf{Proof of the consistency of $\widehat{\sigma}_{1.2 \ldots p}^{2}$}.
Let $\bs{\xi}_{1s}$ denote an $n  \times 1$ vector of $\xi_1$ for complete cases. By the law of large numbers, the continuous mapping theorem and regularity assumptions \ref{ThA:1}, \ref{ThA:2},
\begin{flalign}
\widehat{\sigma}_{1.2 \ldots p, ols}^2&=\frac{1}{n-p}\bs{X}_{1s}^T(\bs{I}_n-\bs{X}_{\m1s}(\bs{X}_{\m1s}^T\bs{X}_{\m1s})^{\m1}\bs{X}_{\m1s}^T)\bs{X}_{1s}=\frac{\bs{\xi}_{1s}^T\bs{\xi}_{1s}}{n-p}-\frac{\bs{\xi}_{1s}^T\bs{X}_{\m1s}(\bs{X}_{\m1s}^T\bs{X}_{\m1s})^{\m1}\bs{X}_{\m1s}^T\bs{\xi}_{1s}}{n-p}  \nonumber &\\
& \rightarrow_p E(\xi_{1}^{2}|Z=1)-E(\xi_{1}\bs{X}_{\m1}|Z=1)\E^{\m1}(\bs{X}_{\m1}^{T} \bs{X}_{\m1}|Z=1)E(\xi_{1}\bs{X}^{T}_{\m1}|Z=1). \label{eq:thA1}&
\end{flalign}
Since for MDM A $\E\epsilon=0$, $\eta_1 \indep  (X_2,\ldots, X_p)$ and $\epsilon \indep ( X_2,\ldots, X_p,\eta_1),$
\begin{flalign}
\E(\xi_1^{2}|Z=1)&=\gamma^{2}\sigma_{1.2\ldots p}^{2}\text{E}(\eta_1^{2}|Z=1)+2\gamma\sigma_{1.2\ldots p}\text{E}(\eta_1\epsilon|Z=1)+\text{E}(\epsilon^{2}|Z=1)\nonumber & \\
& = \gamma^{2}\sigma_{1.2\ldots p}^{2}\E\left[\E(\eta_1^{2}|\bs{X}_{\m1},Z=1)|Z=1\right]+\E\epsilon^{2}=\gamma^{2}\sigma_{1.2\ldots p}^{2}\E(1+u\lambda_u|Z=1)+\sigma_{1.2\ldots p}^{2}(1-\gamma^{2}), \label{eq:thA2} &
\end{flalign}
\begin{flalign} 
 &\E(\xi_1\bs{X}_{\m1}|Z=1) = \E\left[\bs{X}_{\m1}\E(\gamma\sigma_{1.2\ldots p}\eta_1+\epsilon|\bs{X}_{\m1},Z=1)|Z=1\right]=  \gamma\sigma_{1.2\ldots p}\E(\bs{X}_{\m1}\lambda_u|Z=1) \label{eq:thA3}.&
\end{flalign}
The last equality in (\ref{eq:thA2}) follows from the expressions for the variance and the mean of truncated normal distribution (see \citet{heckman1979sample}):
$\E\left[\E(\eta_1^{2}|\bs{X}_{\m1},Z=1)|Z=1\right]=\E\left[\E(\eta_1^{2}|\bs{X}_{\m1},\eta_1 > -\bs{X}_{\m1}\bs{\delta})|Z=1\right]=\E(1+u\lambda_{u}|Z=1).$
From (\ref{eq:thA1}), (\ref{eq:thA2}) and (\ref{eq:thA3}), 
 \begin{flalign} 
&\widehat{\sigma}_{1.2 \ldots p, ols}^2\rightarrow_{p} \sigma_{1.2\ldots p}^{2}\left[1+\gamma^2\E(u\lambda_u|Z=1)-\gamma^2\E\left(\bs{X}_{\m1}\lambda_u|Z=1\right)\E^{\m1}(\bs{X}_{\m1}^{T} \bs{X}_{\m1}|Z=1)\E(\bs{X}_{\m1}^{T}\lambda_u|Z=1)\right].&\label{eq:thA4}
\end{flalign}
Since $\bs{\widehat{\delta}}$ is the consistent estimator in probit regression, 
$\widehat{u}_i$ and $\lambda_{\widehat{u}_i}$ are consistent estimators of $u_i$ and $\lambda_{u_i}$.  $E\lambda^{2}_u < \infty $, since similar to \citet{birnbaum1942aninequality} it can be shown that $\forall x \in R$  $\lambda(x)\leq 2|x|+2$. From the regularity assumption \ref{ThA:1} by the law of large numbers and the continuous mapping theorem,
\begin{flalign}
&\bs{\widehat{u}}^{T}\bs{\lambda}_{\widehat{\bs{u}}}/(n-p)= (\bs{\widehat{u}}^{T}\bs{\lambda}_{\widehat{\bs{u}}} - \bs{u}^{T}\bs{\lambda}_{\bs{u}})/(n-p)+\bs{u}^{T}\bs{\lambda}_{\bs{u}}/(n-p)\rightarrow_{p} 0+\E(u\lambda_u|Z=1), \label{eq:thA5}&\\
&\bs{\lambda}_{\widehat{\bs{u}}}^{T}\bs{X}_{\m1s} (\bs{X}_{\m1s}^T\bs{X}_{\m1s})^{\m1} \bs{X}_{\m1s}^T \bs{\lambda}_{\widehat{\bs{u}}}/(n-p)\rightarrow_{p} \E\left(\bs{X}_{\m1}\lambda_u|Z=1\right)\E^{\m1}(\bs{X}_{\m1}^{T}\bs{X}_{\m1}|Z=1)\E(\bs{X}_{\m1}^{T}\lambda_u|Z=1).\label{eq:thA6}&
\end{flalign} 
From (\ref{eq:thA4}), (\ref{eq:thA5}), (\ref{eq:thA6}), regularity assumption \ref{ThA:4} and the continuous mapping theorem
\begin{flalign}& \widehat{\sigma}_{1.2\ldots p}^{2} \rightarrow_{p} \sigma_{1.2\ldots p}^{2}. \label{eq:thA7}&
\end{flalign}
\textbf{Consistency of $\widehat{\beta}_{2}$} follows from the law of large numbers, the continuous mapping theorem, regularity assumptions \ref{ThA:1}, \ref{ThA:2} and  (\ref{eq:thA3}):
\begin{flalign}
\widehat{\beta}_2=&(\bs{X}_{\m1s}^T\bs{X}_{\m1s})^{\m1}\bs{X}_{\m1s}^T[\bs{X}_{1s}-\gamma\widehat{\sigma}_{1.2 \ldots p}\bs{\lambda}_{\widehat{\bs{u}}}]e_{2} \nonumber \\
& \rightarrow_p \left[\E^{\m1}(\bs{X}_{\m1}^{T} \bs{X}_{\m1}|Z=1) \left(\E(\bs{X}_{\m1}^{T}\bs{X}_{\m1}\bs{\beta}|Z=1)+\E(\bs{X}_{\m1}^{T}\xi_1|Z=1)-\gamma\sigma_{1.2\ldots p}\E(\bs{X}_{\m1}^{T}\bs{\lambda}_{\bs{u}}|Z=1)\right)\right]_2=\beta_2, \label{eq:thA8}&
\end{flalign}
where $\bs{\beta}=(\beta_1, \ldots, \beta_p)^{T}.$

\textbf{Asymptotic normality of $\widehat{\beta}_{2}$} follows from Slutsky's and the multivariate central limit theorem,
\begin{flalign*}
\sqrt{n}(\widehat{\beta}_2-\beta_2)&=\sqrt{n}\left[(\bs{X}^T_{\m1s}\bs{X}_{\m1s})^{\m1}[\bs{X}^T_{\m1s}\bs{\xi}_{1s}-\gamma\widehat{\sigma}_{1.2 \ldots p}\bs{X}_{\m1s}^T \bs{\lambda}_{\widehat{\bs{u}}}]\right]_2&\\
&\rightarrow_{d}\mathcal{N}\left(0,\left[\E^{\m1}(\bs{X}_{\m1}^{T} \bs{X}_{\m1}|Z=1)\var(\bs{X}^T_{\m1}(\xi_{1}-\gamma\sigma_{1.2\ldots p}\lambda_u)|Z=1)\E^{\m1}(\bs{X}_{\m1}^{T} \bs{X}_{\m1}|Z=1)^{T}\right]_{22}\right).
\end{flalign*}
From the properties of truncated normal distribution,
\begin{flalign*} 
\var(\bs{X}^T_{\m1}(\xi_{1}-\gamma\sigma_{1.2\ldots p}\lambda_u)|Z=1)=&\E\left(\bs{X}^T_{\m1}\var(\xi_1-\gamma\sigma_{1.2\ldots p}\lambda_u|\bs{X}_{\m1},Z=1)\bs{X}_{\m1}|Z=1\right)&\\
&+\var\left(\E(\bs{X}^T_{\m1}\bs{X}_{\m1}|Z=1)\bs{\beta} \right)&\\
= &\E\left(\bs{X}^T_{\m1} \left[\gamma^{2}\sigma_{1.2\ldots p}^{2}\var(\eta_1|\bs{X}_{\m1},Z=1) + \var(\epsilon|\bs{X}_{\m1},Z=1)\right]\bs{X}_{\m1}|Z=1\right)&\\
 = &\E\left(\bs{X}^T_{\m1} \left[\gamma^{2}\sigma_{1.2\ldots p}^{2}(1+u\lambda_u-\lambda_u^2)+ \sigma_{1.2\ldots p}^2 (1-\gamma^{2})\right]\bs{X}_{\m1}|Z=1\right)&\\
= & \sigma_{1.2\ldots p}^{2}\E\left(\bs{X}^T_{\m1} \bs{X}_{\m1}\left[(1+\gamma^{2}u\lambda_u-\gamma^{2}\lambda_u^2)\right]|Z=1\right).&
\end{flalign*}
Therefore, 
\begin{flalign}
&\sqrt{n}(\widehat{\beta}_2-\beta_2) \rightarrow_{d} \mathcal{N}\left(0,\nu_{\beta_2}\right)\label{eq:thA9}\\
& \nu_{\beta_2}=\sigma_{1.2\ldots p}^{2}[\E^{\m1}(\bs{X}_{\m1}^{T}\bs{X}_{\m1}|Z=1)\E\left(\bs{X}^T_{\m1} \bs{X}_{\m1}\left[(1+\gamma^{2}u\lambda_u-\gamma^{2}\lambda_u^2)\right]|Z=1\right)\E^{\m1}(\bs{X}_{\m1}^{T}\bs{X}_{\m1}|Z=1)^{T}]_{22}.\nonumber &
\end{flalign}
From the consistency of $\widehat{\sigma}_{2.3\ldots p}^{2}$, (\ref{eq:thA7}), (\ref{eq:thA8}), regularity assumption \ref{ThA:5} and the continuous mapping theorem, $\widehat{\rho}_{\gamma} \rightarrow_p \rho$. Using additionally (\ref{eq:thA9}) and Slutsky's theorem, it follows that
\begin{flalign*}
&\sqrt{n}(\widehat{\rho}_{\gamma}-\rho) \rightarrow_d \mathcal{N}\left(0,\frac{\nu_{\beta_2}}{\beta^2_{2}+\sigma^2_{1.2\ldots p}/\sigma^2_{2.3\ldots p}} \right).&
\end{flalign*}
\textbf{Consistency of $\widehat{\text{se}}_{\widehat{\rho}_{\gamma}}$} follows from the law of large numbers and the continuous mapping theorem: 
\begin{flalign*}
&\widehat{\text{se}}_{\widehat{\rho}_{\gamma}}^{2}=\frac{\widehat{\sigma}_{1.2 \ldots p}^{2}(1+\gamma^{2}\bs{\widehat{u}}^{T}\bs{\lambda}_{\bs{\widehat{u}}}/n-\gamma^{2}\bs{\lambda}_{\bs{\widehat{u}}}^{T}\bs{\lambda}_{\bs{\widehat{u}}}/n) (\bs{X}_{\m1s}^T\bs{X}_{\m1s})_{22}^{\m1}}{\widehat{\beta}_2^2+\widehat{\sigma}_{1.2 \ldots p}^2/\widehat{\sigma}^2_{2.3\ldots p}}\rightarrow_p \frac{\nu_{\beta_2}}{\beta^2_{2}+\sigma^2_{1.2\ldots p}/\sigma^2_{2.3\ldots p}}.&
\end{flalign*}
Therefore, $$\sqrt{n}(\widehat{\rho}_{\gamma}-\rho)/\widehat{\text{se}}_{\widehat{\rho}_{\gamma}} \rightarrow_d \mathcal{N}(0, 1). $$
\end{proof}

\begin{proof}[Proof of Corollary \ref{corollary: coverage}]
Let denote $\Gamma=[\gamma_{min}, \gamma_{max}].$
\begin{flalign}
\text{pr}\left(\rho_0  \in UR(\rho,\Gamma, \alpha \right)&=\text{pr}\left(\rho_0  \in \bigcup\limits_{ \gamma \in \Gamma}\text{CI}(\rho, \gamma,\alpha) \right)= 1-\text{pr}\left(\rho_0  \notin \bigcup\limits_{ \gamma \in \Gamma}\text{CI}(\rho, \gamma,\alpha) \right) \nonumber\\
&=1-\text{pr}\left(\rho_0  \in \bigcap\limits_{ \gamma \in \Gamma} (-\infty,\widehat{\rho}_{\gamma}-c_{\frac{\alpha}{2}}\widehat{\text{se}}_{\widehat{\rho}_{\gamma}}) \cup (\widehat{\rho}_{\gamma}+c_{\frac{\alpha}{2}}\widehat{\text{se}}_{\widehat{\rho}_{\gamma}}, \infty ) \right)\nonumber\\ 
&=1-\text{pr}(\bigcap\limits_{\gamma \in \Gamma}\{\widehat{\rho}_{\gamma}-c_{\frac{\alpha*}{2}}\widehat{\text{se}}_{\widehat{\rho}_{\gamma}} > \rho_0\})- \text{pr}(\bigcap\limits_{\gamma \in \Gamma}\{\widehat{\rho}_{\gamma}+c_{\frac{\alpha*}{2}}\widehat{\text{se}}_{\widehat{\rho}_{\gamma}} < \rho_0\}) \nonumber&\\
&\geq  1-\text{pr}(\widehat{\rho}_{\gamma_0}-c_{\frac{\alpha*}{2}}\widehat{\text{se}}_{\widehat{\rho}_{\gamma_0}}> \rho_0)-\text{pr}(\widehat{\rho}_{\gamma_0}+c_{\frac{\alpha*}{2}}\widehat{\text{se}}_{\widehat{\rho}_{\gamma_0}} < \rho_0) \label{col1} &\\
&=\text{pr}(\widehat{\rho}_{\gamma_0}-c_{\frac{\alpha*}{2}}\widehat{\text{se}}_{\widehat{\rho}_{\gamma_0}}< \rho_0< \widehat{\rho}_{\gamma_0}+c_{\frac{\alpha*}{2}}\widehat{\text{se}}_{\widehat{\rho}_{\gamma_0}}) \rightarrow 1-\alpha,\  n\rightarrow \infty \nonumber&
\end{flalign}
from the definition of CI. (\ref{col1}) follows since $
\bigcap\limits_{\gamma \in \Gamma}\{\widehat{\rho}_{\gamma}-c_{\frac{\alpha*}{2}}\widehat{\text{se}}_{\widehat{\rho}_{\gamma}} > \rho_0\} \subset \{\widehat{\rho}_{\gamma_0}-c_{\frac{\alpha*}{2}}\widehat{\text{se}}_{\widehat{\rho}_{\gamma_0}}> \rho_0\}$ and $\bigcap\limits_{\gamma \in \Gamma}\{\widehat{\rho}_{\gamma}-c_{\frac{\alpha*}{2}}\widehat{\text{se}}_{\widehat{\rho}_{\gamma}} < \rho_0\} \subset \{\widehat{\rho}_{\gamma_0}+c_{\frac{\alpha*}{2}}\widehat{\text{se}}_{\widehat{\rho}_{\gamma_0}} < \rho_0\} $.

\end{proof}
 
\section{}
\label{app:B}
\begin{theorem} \label{theorem: CAN for rho mechanism B}

Let $\{X_{1i}, X_{2i}, \ldots, X_{pi},Z_i\}_{i=1}^{N}$  be a random sample from $(X_{1}, X_{2}, \ldots, X_{p},Z)$ for which MDM B holds. Under regularity assumptions:
 \begin{enumerate}
\item $EX_j^{2}<~\infty, \  j=1,\ldots,p;$ $\sigma^2_{2.3\ldots p}\neq 0$, $\sigma^2_{1.2\ldots p}\neq 0;$
 \item $ \det \bs{X}_{\m1s}^T\bs{X}_{\m1s}\neq0$, $ \det \E(\bs{X}_{\m1}^{T} \bs{X}_{\m1}|Z=1) \neq0,$ where $\bs{X}_{\m1s}$  is an $n \times p$ matrix of observed covariates $(1, X_2, \ldots, X_p)$ for complete cases, $\bs{X}_{\m1}=(1, X_2, \ldots, X_p)$ is a vector of random variables; 
 \item $\det\bs{X}_{\m12s}^T\bs{X}_{\m12s}\neq0, \  \det \E(\bs{X}_{\m12}^{T} \bs{X}_{\m12}) \neq0,$ where $\bs{X}_{\m12s}$ - $n \times (p-1)$ matrix of observed covariates $(1, X_3, \ldots, X_p)$ for complete cases and $\bs{X}_{\m12}=(1, X_3, \ldots, X_p)$ is a vector of random variables;
 \item  $1+\frac{\gamma^2}{n-p}(\bs{\widehat{u}}^{T}\bs{\lambda}_{\widehat{\bs{u}}}-\bs{\lambda}_{\widehat{\bs{u}}}^{T}\bs{X}_{\m1s} (\bs{X}_{\m1s}^T\bs{X}_{\m1s})^{\m1} \bs{X}_{\m1s}^T \bs{\lambda}_{\widehat{\bs{u}}}) \neq 0 $ for all $n$, \\$
 1+\gamma^2\E(u\lambda_u|Z=1)-\gamma^2\E\left(\bs{X}_{\m1}\lambda_u|Z=1\right)\E^{\m1}(\bs{X}_{\m1}^{T}\bs{X}_{\m1}|Z=1)\E(\bs{X}_{\m1}^{T}\lambda_u|Z=1)\neq 0,$
  where $u=-\bs{X}_{\m12}\bs{\delta}$ and $\lambda_u=\frac{\phi(u)}{\Phi(- u)}$, $\lambda$ denotes inverse Mills ratio;
 \item  $\widehat{\sigma}^2_{2.3\ldots p}\neq 0$, and at least one of $\widehat{\beta}_{2} \neq 0$ or $\widehat{\sigma}_{1.2\ldots p}^2 \neq 0$ for all $n;$
 \end{enumerate}
 $\widehat{\rho}_{\gamma}$ is a consistent estimator of $\rho$ and
$$\sqrt{n}(\widehat{\rho}_{\gamma}-\rho)/\widehat{\text{se}}_{\widehat{\rho}_{\gamma}} \rightarrow_d \mathcal{N}(0, 1), $$
where
\begin{align*}
&\widehat{\rho}_{\gamma}=\frac{\widehat{\beta}_{2}}{\sqrt{\widehat{\beta}_{2}^2+\widehat{\sigma}_{1.2\ldots p}^2/\widehat{\sigma}^2_{2.3\ldots p}}},\\&\widehat{\beta}_2=\widehat{\beta}_{2,ols}-\gamma\widehat{\sigma}_{1.2 \ldots p}\left[(\bs{X}_{\m1s}^T\bs{X}_{\m1s})^{\m1}\bs{X}_{\m1s}^T\bs{\lambda}_{\widehat{\bs{u}}}\right]_2,\\
&\widehat{\sigma}_{1.2 \ldots p}^2=\frac{\widehat{\sigma}_{1.2 \ldots p, ols}^2}{1+\gamma^2(\bs{\widehat{u}}^{T}\bs{\lambda}_{\widehat{\bs{u}}}-\bs{\lambda}_{\widehat{\bs{u}}}^{T}\bs{X}_{\m1s} (\bs{X}_{\m1s}^T\bs{X}_{\m1s})^{\m1} \bs{X}_{\m1s}^T \bs{\lambda}_{\widehat{\bs{u}}})/(n-p) },\\
&\widehat{\text{se}}_{\widehat{\rho}_{\gamma}}=\sqrt{\frac{\widehat{\sigma}_{1.2 \ldots p}^{2}(1+\gamma^{2}\bs{\widehat{u}}^{T}\bs{\lambda}_{\bs{\widehat{u}}}/n-\gamma^{2}\bs{\lambda}_{\bs{\widehat{u}}}^{T}\bs{\lambda}_{\bs{\widehat{u}}}/n) (\bs{X}_{\m1s}^T\bs{X}_{\m1s})_{22}^{\m1}}{\widehat{\beta}^2_2+\widehat{\sigma}^2_{1.2 \ldots p}/\widehat{\sigma}^2_{2.3\ldots p}}},
\end{align*}
and $\widehat{\sigma}^2_{2.3\ldots p}$ is an ordinary OLS estimator of $\sigma^2_{2.3\ldots p}$ based on complete cases $\{X_{2i}, \ldots, X_{pi}\}_{i=1}^{n}.$ $\bs{X}_{1s}=(X_{11}, \ldots, X_{1n})^{T}$ denotes an $n<N$ vector of observed $X_{1};$ $\bs{\widehat{u}}=(\widehat{u}_1, \ldots,\widehat{u}_n)^{T} =-\bs{X}_{\m12s}\widehat{\bs{\delta}}$, $\widehat{\bs{\delta}}$ is the maximum likelihood estimator of $\bs{\delta}$ in probit model for missingness in mechanism B based on full data $\{X_{3i}, \ldots, X_{pi},Z_i\}_{i=1}^{N}$, $\bs{\lambda_{\widehat{u}}}=(\lambda(\widehat{u}_1), \ldots, \lambda(\widehat{u}_n))^{T}$, $\lambda(\widehat{u}_i)=\frac{\phi(\widehat{u}_i)}{\Phi(-\widehat{u}_i)}$  for all $ i=1,\ldots, N$ denotes the inverse Mills ratio, $\phi$ and $\Phi$ are, respectively, the standard normal density and cumulative distribution functions. Also, $[v]_2$ denotes the second element of a vector $v.$  
 \begin{proof} Similar to the proof of \autoref{theorem CAN},
 \begin{align*}
&\E(\xi^2_{1}|Z=1)=\E\left[\E(\xi^2_{1}|\bs{X}_{\m12}, Z=1)|Z=1\right]=\E\left[\E\left((\gamma_{1}\sigma_{1.2\ldots p}\eta_{1}+\epsilon)^{2}|\bs{X}_{\m12}, \eta_1>\m\bs{X}_{\m12}\bs{\delta}\right)|\eta_1>\m\bs{X}_{\m12}\bs{\delta}\right]\\
&=\sigma_{1.2\ldots p}^{2}\E\left(1+\gamma^{2}u\lambda_{u}|Z=1\right),\\
&\E(\bs{X}_{\m1}\xi_1|Z=1)=\E\left[\bs{X}_{\m1}\E(\gamma\sigma_{1.2\ldots p}\eta_1+\epsilon|\bs{X}_{\m1},Z=1)|Z=1\right]=\gamma\sigma_{1.2\ldots p}\E(\bs{X}_{\m1}\lambda_u|Z=1).
\end{align*}
Let $\bs{\xi}_{2s}$ denote an $n$ vector of $\xi_2$ for complete cases. Since $\xi_2 \indep (X_3,\ldots, X_p, \eta_1)$ and $\eta_1 \indep (X_2,\ldots, X_p),$
\begin{align*}
&\E\widehat{\sigma}_{2.3 \ldots p}^2=\E\frac{1}{n-p}\bs{X}_{2s}^T(\bs{I}_n-\bs{X}_{\m12s}(\bs{X}_{\m12s}^T\bs{X}_{\m12s})^{\m1}\bs{X}_{\m12s}^T)\bs{X}_{2s}\\
&=\frac{1}{n-p}trE\left[\bs{M}_{2s}E\left[\bs{\xi}_{2s}\bs{\xi}_{2s}^T|\bs{X_{\m12s},\bs{\eta}_1>\m\bs{X}_{\m12}\bs{\delta}) }\right]\right]=\frac{1}{n-p}trE\left[\bs{M}_{2s}E\left[\bs{\xi}_{2s}\bs{\xi}_{2s}^T\right]\right]=\sigma_{2.3\ldots p}^{2}.
\end{align*}
The remaining parts of the proof follow the proof of theorem \autoref{theorem CAN}. 
 \end{proof}
 \end{theorem}
 A $(1-\alpha)100\%$ confidence interval for $\rho$ is
$
\text{CI}(\rho,\gamma,\alpha)=[\widehat{\rho}_{\gamma}-c_{\frac{\alpha}{2}}\widehat{\text{se}}_{\widehat{\rho}_{\gamma}};\widehat{\rho}_{\gamma}+c_{\frac{\alpha}{2}}\widehat{\text{se}}_{\widehat{\rho}_{\gamma}}].
$
Here $c_{\frac{\alpha}{2}}$ is the  $(1-\alpha)100\%$  percentile of the standard normal distribution.

\begin{corollary}
\label{corollary: coverage B}
Under the assumptions of \autoref{theorem: CAN for rho mechanism B}, if the true $\gamma_0 \in [\gamma_{min}, \gamma_{max}]$, the uncertainty region $\text{UR}(\rho,[\gamma_{min}, \gamma_{max}], \alpha )=\bigcup\limits_{ \gamma \in [\gamma_{min}, \gamma_{max}]} \text{CI}(\rho,\gamma,\alpha)$ has asymptotic coverage for $\rho_0$  of at least $(1-\alpha)100\%.$
\begin{proof}
Proof follows the same structure as the proof of \autoref{corollary: coverage}.
\end{proof}
\end{corollary}

\begin{theorem} \label{theorem: CAN for rho mechanism C}
 Let $\{X_{1i}, X_{2i}, \ldots, X_{pi},Z_{1i},Z_{2i}\}_{i=1}^{N}, $ be a random sample from $(X_{1}, X_{2}, \ldots, X_{p},Z_1,Z_2)$   for which MDM C holds. Under regularity assumptions:
\begin{enumerate}
 \item  $EX_j^{2}<~\infty, \  j=1,\ldots,p;$ $\sigma^2_{2.3\ldots p}\neq 0$, $\sigma^2_{1.2\ldots p}\neq 0.$
 \item  $ \det \bs{X}_{\m1s}^T\bs{X}_{\m1s}\neq0$, $ \det \E(\bs{X}_{\m1}^{T} \bs{X}_{\m1}|Z_1=1, Z_2=1) \neq0$ where $\bs{X}_{\m1s}$  is an $n \times p$ matrix of observed covariates $(1, X_2, \ldots, X_p)$ for complete cases, $\bs{X}_{\m1}=(1, X_2, \ldots, X_p)$ is a vector of random variables;  
 \item  $\det\bs{X}_{\m12s_2}^T\bs{X}_{\m12s_2}\neq0$, $ \det \E(\bs{X}_{\m12}^{T} \bs{X}_{\m12}| Z_2=1) \neq0$, where $\bs{X}_{\m12s_2}$ - $n_2 \times (p-1)$ matrix of observed covariates $(1, X_{3}, \ldots, X_{p})$ for cases with observed $X_2$ and $\bs{X}_{\m12}=(1, X_3, \ldots, X_p)$ is a vector of random variables; 
 \item  $1+\frac{\gamma_1^2}{n-p}(\bs{\widehat{u}}^{T}\bs{\lambda}_{\widehat{\bs{u}}}-\bs{\lambda}_{\widehat{\bs{u}}}^{T}\bs{X}_{\m1s} (\bs{X}_{\m1s}^T\bs{X}_{\m1s})^{\m1} \bs{X}_{\m1s}^T \bs{\lambda}_{\widehat{\bs{u}}})\neq 0 $ for all $ n,$\\ $
  1+\gamma_1^2\E(u\lambda_u|Z_1=1, Z_2=1)-\gamma_1^2\E\left(\bs{X}_{\m1}\lambda_u|Z_1=1, Z_2=1\right) \times\\
  \times \E^{\m1}\left(\bs{X}_{\m1}^{T}\bs{X}_{\m1}|Z_1=1, Z_2=1\right)\E\left(\bs{X}_{\m1}^{T}\lambda_u|Z_1=1, Z_2=1\right)\neq 0,
 $ where $u=-\bs{X}_{\m12}\bs{\delta}_1$ and $\lambda_u=\frac{\phi(u)}{\Phi(- u)}$, $\lambda$ denotes inverse Mills ratio.
 \item  $1+\frac{\gamma_2^2}{n_2-p} (\bs{\widehat{w}}^{T}\bs{\lambda}_{\widehat{\bs{w}}}-\bs{\lambda}_{\widehat{\bs{w}}}^{T}\bs{X}_{\m12s_2} (\bs{X}_{\m12s_2}^T\bs{X}_{\m12s_2})^{\m1} \bs{X}_{\m12s_2}^T \bs{\lambda}_{\widehat{\bs{w}}})\neq 0$ for all $n_2,$\\ $
  1+\gamma_2^2\E(w\lambda_w| Z_2=1)- \gamma_2^2\E\left(\bs{X}_{\m1}\lambda_w|Z_2=1\right)
  \E^{\m1}\left(\bs{X}_{\m1}^{T}\bs{X}_{\m1}| Z_2=1\right)\E\left(\bs{X}_{\m1}^{T}\lambda_w|Z_2=1\right)\neq 0,
 $ where $w=-\bs{X}_{\m12}\bs{\delta}_2$ and $\lambda_w=\frac{\phi(w)}{\Phi(- w)};$
 \item $\widehat{\sigma}^2_{2.3\ldots p}\neq 0$ and at least one of $\widehat{\beta}_{2} \neq 0$ or $\widehat{\sigma}_{1.2\ldots p}^2 \neq 0$ for each sample;
 \end{enumerate}
  $\widehat{\rho}_{\gamma_1, \gamma_2}$ is a consistent estimator of $\rho$ and $$\sqrt{n}(\widehat{\rho}_{\gamma_1, \gamma_2}-\rho)/\widehat{\text{se}}_{\widehat{\rho}_{\gamma_1, \gamma_2}} \rightarrow_d \mathcal{N}(0, 1),$$
\begin{proof}
Similar to the proof of \autoref{theorem CAN},
 \begin{align*}
\E(\xi^2_{1}|Z_1=1, Z_2=1)&=\E\left[\E(\xi^2_{1}|\bs{X}_{\m12}, Z_1=1,Z_2=1)|Z_1=1,Z_2=1\right]\\
&=\E\left[\E\left((\gamma_{1}\sigma_{1.2\ldots p}\eta_{1}+\epsilon_{1})^{2}|\bs{X}_{\m12}, \eta_1>\m\bs{X}_{\m12}\bs{\delta}_1,\eta_2>\m\bs{X}_{\m12}\bs{\delta}_2\right)|Z_1=1,Z_2=1\right]\\
&=\sigma_{1.2\ldots p}^{2}\E\left(1+\gamma_1^{2}u\lambda_{u}|Z_1=1,Z_2=1\right),
\end{align*}
\begin{align*}
\E(\bs{X}_{\m1}\xi_1|Z_1=1, Z_2=1)&=\E\left[\bs{X}_{\m1}\E(\gamma_1\sigma_{1.2\ldots p}\eta_1+\epsilon_1|\bs{X}_{\m1},Z_1=1, Z_2=1)|Z_1=1, Z_2=1\right]\\
&=\gamma_1\sigma_{1.2\ldots p}\E(\bs{X}_{\m1}\lambda_u|Z_1=1, Z_2=1).
\end{align*}
The remaining parts of the proof follow the proof of \autoref{theorem CAN} with $u$ defined as $u=-\bs{X}_{\m12}\bs{\delta}_1$ and $\widehat{\bs{\delta}}_1$ as the maximum likelihood estimates of $\bs{\delta}_1$ in probit model based on full data $\{X_{3i}, \ldots, X_{pi},Z_{1i}\}_{i=~1}^{N}.$
The proof of consistency $\widehat{\sigma}_{2.3\ldots p}^{2}$ follows the one for the consistency of  $\widehat{\sigma}_{1.2\ldots p}^{2}$ in theorem \autoref{theorem CAN} given that $\E\epsilon_2=0,$ $\eta_2 \indep (X_3,\ldots, X_p)$ , $\epsilon_2 \indep (X_3,\ldots, X_p,\eta_2).$
\end{proof}
\end{theorem}

\begin{corollary}
\label{corollary: coverage C}
Under the assumptions of \autoref{theorem: CAN for rho mechanism C}, if the true $\gamma_{10} \in [\gamma_{1min}, \gamma_{1max}],$ $\gamma_{20} \in [\gamma_{2min}, \gamma_{2max}]$, the uncertainty region $\text{UR}(\rho,[\gamma_{1min}, \gamma_{1max}],[\gamma_{2min}, \gamma_{2max}], \alpha)$ has asymptotic coverage for $\rho_0$  of at least $(1-\alpha)100\%.$
\begin{proof}
Proof follows the same structure as the proof of \autoref{corollary: coverage}.
\end{proof}
\end{corollary}

\end{appendices}
\section*{}
\bibliography{PartCorArticleBibliography}

\end{document}